\documentclass{amsart}
\usepackage{amsthm,amssymb,amsfonts,setspace,verbatim,graphicx,mathtools,mathrsfs,commath,float}

\usepackage{hyperref}
\usepackage{ulem}
\usepackage[margin=1.62in]{geometry}
\newtheorem{ut}{Theorem}
\newtheorem*{utm}{Theorem}

\newtheorem{ul}[ut]{Lemma}

\newtheorem*{uq}{Question}
\theoremstyle{remark}
\newtheorem*{ur}{Remark}
\begin{document}

\title{Large widely-connected spaces}
\author{David Lipham}
\address{Department of Mathematics and Statistics, Auburn University, Auburn,
AL 36849, United States}
\email{dsl0003@auburn.edu}

\begin{abstract}We show that  there are widely-connected spaces of arbitrarily large cardinality, answering a question by David  Bellamy.  \end{abstract}

\maketitle

A connected space $W$ is said to be \textit{widely-connected} if every non-degenerate connected subset of $W$ is dense in $W$. %The term was originally  defined by Paul Swingle \cite{3}, who constructed the first example.  
In \cite{1}, David Bellamy asked whether there are widely-connected sets of arbitrarily large cardinality.  We prove the following.
%In this paper we use the main results of [] and [] to prove the following.

\begin{utm}\label{main}For each infinite cardinal $\kappa$ there is a completely regular widely-connected space of size $2^\kappa$.\end{utm}  

%\noindent 

In \cite{3}, Paul Swingle showed how to construct widely-connected subsets of metric indecomposable continua (a \textit{continuum} is a connected compact Hausdorff space, and if $X$ is a continuum then $X$ is \textit{indecomposable} if $X$ is not the union of two proper subcontinua), and in \cite{2}, Michel Smith showed that there are indecomposable continua of arbitrarily large cardinality. Swingle's method is generalized by Lemma \ref{dwc} of this paper, so that it can be applied to Smith's non-metric constructions.

To prove Lemma \ref{dwc} we will use some basic properties of  composants. If $X$ is a continuum then $C$ is a \textit{composant} of $X$ if there exists $x\in X$ such that $C$ is the union of all proper subcontinua of $X$ which contain $x$. Two standard results of continuum theory are (1) if $X$ is a continuum then each composant of $X$ is connected and dense in $X$, and (2) the composants of an indecomposable continuum are pairwise disjoint.  For proofs, see \cite{kur} \S 48 VI.

%abs??

\begin{ul}\label{dwc} Let $X$ be an indecomposable continuum. Let  $\mathscr C(X)$ be the set of composants of $X$, and let $$\mathscr S(X) =\{S\subseteq X:S\text{ is closed and }X\setminus S \text{ is not connected}\}.$$ If $|\mathscr S(X)|\leq |\mathscr C(X)|$ then $X$ has a dense widely-connected subspace of size $|\mathscr C(X)|$.\end{ul}

\begin{proof}Write $\mathscr S$ for $\mathscr S (X)$ and $\mathscr C$ for $\mathscr C (X)$. Since each member of $\mathscr C$ is connected and dense in $X$, we have $C\cap S\neq\varnothing$ for each $C\in\mathscr C$ and $S\in\mathscr S$. Let $\varPsi:\mathscr C \to \mathscr S$ be a surjection, and for each $C\in\mathscr C$ let $\psi(C)\in C\cap \varPsi(C)$. Let $W=\{\psi(C):C\in\mathscr C\}$. This construction is made possible by the Axiom of Choice. Note that $|W|=|\mathscr C|$ since the members of $\mathscr C$ are pairwise disjoint. 

Let $U$ be a nonempty open subset of $X$.  Let $V\subseteq X$ be nonempty and open such that $\overline V\subseteq U$ and $\overline V \neq X$. Then $\partial V\in \mathscr S$. There exists $C\in\mathscr C$ such that $\varPsi(C)=\partial V$. Then $\psi(C)\in W\cap \overline V\subseteq W\cap U$.  We have proven that $W$ is dense in $X$. Supposing that $W$ is not connected, there are nonempty open subsets $U$ and $V$ of $X$ such that $U\cap W\neq \varnothing$, $V\cap W\neq\varnothing$,  and $W=(U\cap W)\cup (V\cap W)$. By density of $W$ we have $U\cap V=\varnothing$. So $X\setminus (U\cup V)\in \mathscr S$. By definition $W\cap [X\setminus (U\cup V)]\neq \varnothing$, contrary to $W\subseteq U\cup V$.  Thus $W$ is connected. If $A$ is a non-dense  connected subset of $W$, then $\overline A$ is a proper subcontinuum of $X$ and  is therefore contained in a composant $C$.  By the definition of $W$ and the fact that the composants of $X$ are pairwise disjoint, $\abs{W\cap C}=1$, so $A$ is degenerate. Therefore $W$ is \textit{widely}-connected. \end{proof}

\begin{ur}The assumption $|\mathscr S(X)|\leq |\mathscr C(X)|$ is actually equivalent to $|\mathscr S(X)|= |\mathscr C(X)|$, and so $\varPsi$ can be taken to be a bijection.  Indeed,  let $\lambda=|\mathscr C (X)|$ and let $U$ and $V$ be nonempty open subsets of $X$ with $\overline U\cap \overline V=\varnothing$. $|X|\geq \lambda$ implies that  $|X\setminus \partial U|\geq \lambda$ or $|X\setminus \partial V|\geq \lambda$. Assuming that $|X\setminus \partial U|\geq \lambda$,  let $\mathscr S'(X)=\{\partial U \cup \{x\}:x\in X\setminus \partial U\}$.  Then $\mathscr S'(X)\subseteq \mathscr S (X)$ implies $|\mathscr S(X)|\geq\lambda$.\end{ur}

%\

\begin{ul}\label{r}Let $\kappa$ be an infinite cardinal.  If $\{X_\alpha:\alpha<\kappa\}$ is a collection of spaces each with a basis of  size $\leq \kappa$, then $\prod _{\alpha<\kappa} X_\alpha$ and $$X:=\varprojlim \{X_\alpha:\alpha<\kappa\}$$   have bases of size $\leq \kappa$, and $X$ has $\leq 2^\kappa$ closed subsets.\end{ul}

\begin{proof}For each $\alpha<\kappa$ let $\mathscr B_\alpha$ be a basis for $X_\alpha$ with $\abs{ \mathscr B_\alpha}\leq \kappa$. For each  $f\in \big[\bigcup _{\alpha<\kappa} \{\alpha\}\times \mathscr B_\alpha \big]^{<\omega}$ let $U_f=\{x\in \prod _{\alpha<\kappa}X_\alpha:x(f(i)(0))\in f(i)(1)\text{ for each }i\in\text{dom}(f)\}$. Then $$\mathscr B=\Big\{U_f:f\in\Big[\bigcup _{\alpha<\kappa} \{\alpha\}\times \mathscr B_\alpha \Big]^{<\omega}\Big\}$$ is a basis for $\prod _{\alpha<\kappa} X_\alpha$ with $\abs{\mathscr B}\leq \kappa^{<\omega}=\kappa$. The inverse limit $X$ also has a basis of size $\leq\kappa$, namely $\{B\cap X:B\in \mathcal B\}$, because it is a subspace of the product.

Define $\varphi :\mathcal P (\mathscr B) \to \tau$ by $\varphi(\mathscr U)=X\cap \bigcup \mathscr U$, where $\tau$ is the topology of $X$. If $U\in\tau$ then letting $\mathscr U=\{B\in\mathscr B:B\cap X\subseteq U\}$ we have $\varphi(\mathscr U)=U$, so that $\varphi$ is surjective. Therefore $\abs{\{X\setminus U:U\in\tau\}}=\abs{\tau}\leq \abs{\mathcal P (\mathscr B)}\leq 2^\kappa$.\end{proof}

%\begin{ut}For each infinite cardinal $\kappa$ there is a completely regular widely-connected space of size $2^\kappa$.\end{ut}  

\begin{proof}[\textbf{\textit{Proof of Theorem.}}]Let $\kappa$ be given. By Theorem 2 of \cite{2}, there is an indecomposable continuum $M$ which has $2^\kappa$ composants.  The continuum $M$ is constructed as an inverse limit of $\kappa$-many continua, each of which has a basis of size $\leq \kappa$. To be more specific, $M=\varprojlim \{M_\alpha:\alpha<\kappa\}$ where $M_\alpha$ is a subcontinuum of   $[0,1]^{\alpha+1}$ for each $\alpha<\kappa$ (see the second and third paragraphs of the proof in \cite{2}  for the successor and limit cases of $\alpha$, respectively). By Lemma \ref{r}, $|\mathscr S(M)|\leq 2^\kappa =|\mathscr C(M)|$. By Lemma \ref{dwc}, $M$ contains a widely-connected space of size $2^\kappa$.\end{proof}

Mazurkiewicz \cite{mazur} proved that every metric indecomposable continuum has $\mathfrak c=2^\omega$ composants. Thus, if $X$ is a metric indecomposable continuum then $\mathscr C (X)=\mathscr S (X)$, and so by Lemma \ref{dwc} $X$ has a dense widely-connected subset.  This technique does not apply in general, as there are (non-metric) indecomposable continua with very few composants.  The Stone-\v{C}ech remainder of $[0,\infty)$ is consistently an indecomposable continuum with only one composant \cite{miod}, and Bellamy \cite{one} constructed indecomposable continua with one and two composants in ZFC.

\begin{uq}Does every indecomposable continuum have a dense widely-connected subset? \end{uq}
%every indecomposable continuum contains a metric indecomposable continuum.


\begin{thebibliography}{HD}

\bibitem{one} D.P. Bellamy, Indecomposable continua with one and two composants, Fund. Math. 101 (1978), no. 2, 129--134. 

\bibitem{1}  D.P. Bellamy, Questions in and out of context, Open Problems in Topology II, Elsevier (2007), 259--262.

\bibitem{kur}K. Kuratowski, Topology Vol. II, Academic Press, 1968.

\bibitem{mazur} S. Mazurkiewicz, Sur les continus indecomposables, Fund. Math. 10 (1927), 305-310.

\bibitem{miod}J. Mioduszewski, On composants of $\beta\mathbb R \setminus \mathbb R$, Proceedings of the Conference on Topology and Measure, I (Zinnowitz, 1974), Part 2, Ernst-Moritz-Arndt Univ., Greifswald, 1978, pp. 257--283. 

\bibitem{2} M. Smith, Generating large indecomposable continua, Pacific J. Math. 62 (1976), no. 2, 587--593.

\bibitem{3} P.M. Swingle,  Two types of connected sets. Bull. Amer. Math. Soc. 37 (1931), no. 4, 254--258.


%\bibitem{4} E.W. Miller. Concerning biconnected sets. Fund. Math., 29 (1937), pp. 123-133.

%\bibitem{5} M.E. Rudin. A connected subset of the plane. Fund. Math., 46 (1958), pp. 15-24.

%\bibitem{6} Gruenhage, Gary. Spaces in Which the Nondegenerate Connected Sets Are the Cofinite Sets. Proceedings of the American Mathematical Society Vol. 122, No. 3 (Nov., 1994) , pp. 911-924.

%\bibitem{7} R.C. Walker, The Stone-\v{C}ech Compactification, Springer,
%New York-Berlin, 1974.



\end{thebibliography}
\end{document}